\documentclass[reqno]{amsart}

\usepackage{amsmath}
\usepackage{amsthm}
\usepackage{amsfonts}
\usepackage{amssymb}
\usepackage{hyperref}
\usepackage{geometry}
\usepackage{layout}
\theoremstyle{plain}
 \newtheorem{theorem}{Theorem}[section]

 \newtheorem{lemma}[theorem]{Lemma}
 \newtheorem{corollary}[theorem]{Corollary}

\theoremstyle{remark}
\newtheorem{remark}[theorem]{Remark}

\theoremstyle{definition}

\begin{document}

\title[A formula for $\zeta(2n+1)$ and some related expressions]
{A formula for $\zeta(2n+1)$ and some related expressions}
\author[Thomas Sauvaget]{Thomas Sauvaget}
%\address{http://thomas1111.wordpress.com/16/08/polylog.html}
\email{thomasfsauvaget@gmail.com} 
%\urladdr{http://thomas1111.wordpress.com}

\thanks{
%\textcopyright~2016 Thomas Sauvaget. 
An erroneous claim of a proof of irrationality of all $\zeta(2n+1)$ in the previous version of this manuscript (arXiv.org/03174v3) has been withdrawn. 
The author apologizes for that claim and wishes to thank the editorial board of PMB for pointing out the error and rejecting the paper. 
What remains of this work is not intented for publication anymore. 
The author is very grateful to the anonymous referee of an earlier version of this work (arXiv.org/03174v2) for (a) many comments that improved 
greatly the readability of the paper, (b) a technical suggestion which allowed the author to prove Theorem \ref{fractionzeta} 
(which had been stated earlier as a conjecture based on numerical observations), (c) inviting the author to resubmit a revised version. 
The author also would like to thank the numerous contributors to useful freely available online knowledge resources, 
in particular the arXiv, Wikipedia, the SagemathCloud, WolframAlpha, and Stack Exchange sites.}

\begin{abstract}
Using a polylogarithmic identity, we express the values of $\zeta$ at odd integers $2n+1$ as integrals over unit $n-$dimensional hypercubes of simple 
functions involving products of logarithms. We also prove a useful property of those functions as some of their variables are raised to a power. 
In the case $n=2$, we prove two closed-form  expressions concerning related integrals. Finally, another family of related one-dimensional integrals is studied. 
\end{abstract}

\maketitle
%\tableofcontents
%\setcounter{tocdepth}{2}

\section{Introduction}

Are all values of the Riemann zeta function irrational numbers when the argument is a positive integer ? 
This question goes back to the XVIIIth century when Euler published in 1755, $n$ being a positive integer, 
that $\zeta(2n)=\frac{(-1)^{n+1}B_{2n}(2\pi)^{2n}}{2(2n)!}$ (where $B_{2n}\in\mathbb{Q}$ is an even Bernoulli number) and 
Lindemann proved in 1882 that $\pi$ is transcendental (hence none of its powers is rational) \cite{Li}\footnote{Lambert proved in 1761 that $\pi$ is irrational \cite{GS}, but this does not imply that all powers of $\pi$ are so.}. 
\bigskip

On the other hand, only in 1978 did Ap\'ery \cite{A} famously 
proved that $\zeta(3)$ is irrational. This was later reproved in a variety of ways by several authors, in particular Beukers \cite{B} who devised 
a simple approach involving certain intergrals over $[0,1]^3$ (which will be recalled in section \ref{irrational}). The reader should consult Fischler's very informative Bourbaki Seminar \cite{F} for more details and references. In the early 2000s, an important work of Rivoal \cite{R} and Ball and Rivoal \cite{BR} 
determined that infinitely many values of $\zeta$ at odd integers are irrational, and the work of Zudilin \cite{Z} 
proved that at least one among $\zeta(5), \zeta(7), \zeta(9)$ and $\zeta(11)$ is irrational. Despite these advances, 
to this day no value of $\zeta(2n+1)$ with $2n+1>3$ is known to be irrational. 
\bigskip 

One dimensional integral formulas for $\zeta(2n+1)$ have been known for a long time, for instance the 1965 monograph of Abramowitz and Stegun \cite{AbSt} gives: 
\[
\zeta(2n+1)=(-1)^{n+1}\frac{(2\pi)^{2n+1}}{2(2n+1)!}\int_0^1B_{2n+1}(x)\cot(\pi x)dx
\]
While it bears a striking structural analogy with Euler's formula for $\zeta(2n)$, it is not obvious how one might try to prove or disprove that these numbers 
are irrational.
\bigskip

On the other hand, multidimensional integral formulas for $\zeta(2n+1)$ are more recent: as mentionned by Baumard in his PhD Thesis \cite{Ba}, 
quoting Zagier \cite{Za}, it is Kontsevich in the early 1990s who found such a type of formula for Multiple Zeta Values, 
which in the case of simple zeta boils down, for any odd or even $k$, to: 
\[
\zeta(k)=\int_0^1 \frac{dx_1}{x_1} \int_0^{x_1}\frac{dx_2}{x_2} \cdots \int_0^{x_{k-2}}\frac{dx_{k-1}}{x_{k-1}} \int_0^{x_{k-1}} \frac{dx_k}{1-x_k}
\]

This is easily proved by expanding the integrand in geometric series and integrating. This can be rewritten more simply 
as a multidimensional integral over a unit hypercube:
\[
\zeta(k)=\int\limits_{[0;1]^k} \frac{dx_1\cdots dx_k}{1-x_1\cdots x_k}
\]

While this is much closer to the type of integrals that Beukers used, it is not clear how it might be adapted directly to prove that zeta is irrational at odd integers. 
\bigskip

One should mention that Brown \cite{Br} has in the past few years outlined a geometric approach to understand the structures involved in 
Beukers's proof of the irrationality $\zeta(3)$ and how this may generalize to other zeta values, see also the recent work of Dupont \cite{D} on that topic. 
\bigskip

In this work, we go along another path and prove polylogarithmic identities which then allow to write each $\zeta(2n+1)$ as an alternating sign $(-1)^{n+1}$ times 
a multiple integral over a $n-$dimensional unit hypercube of certain functions involving logarithms (rather than unsigned integrals over a $(2n+1)-$dimensional unit hypercube as in the previously mentioned formulas). 
These functions are shown to have an interesting property: raising some of the variables to a power 
leads to a fractional multiple of $\zeta(2n+1)$ that belongs to the interval $]0,\zeta(2n+1)[$.  
We also investigate two related families of integrals for $n=2$, for which we establish closed-form  expressions, but show that when used in Beukers's framework 
they fail to produce the irrationality of $\zeta(5)$. Finally, we study another type of functions with better decay properties in a 1-dimensional integral related to the previous formulas.
\bigskip

It is rather curious that these precise identities and integrals seem not to have been considered before, despite their simplicity. 
A search through the literature did not return them (we have used the treatise of Lewin \cite{L} as well as the relevant page 
on {\it functions.wolfram.com}\cite{M}). Identities involving polylogarithms of different degrees are rather scarce, 
all the more so when all variables must be integers, and representations of $\zeta(s)$ as multiple integrals over bounded domains,
including some that have been worked out very recently by Alzer and Sondow \cite{AS}, only go as far as a double integral. 
The idea to consider the formulas presented below came to the author in a fortunate way after studying 
and trying to generalize an integral representation of $\zeta(3)=\frac{1}{2}\int_0^1 \frac{\log(x)\log(1-x)}{x(1-x)}dx $ established by Janous \cite{J} (and mentioned by Alzer and Sondow, 
where the author first learned about it), while the idea of trying to prove the irrationality of $\zeta(5)$ was a reaction to a footnote in 
a section of the fine undergraduate book of Colmez \cite{C} devoted to Nesterenko's proof of the irrationality of $\zeta(3)$.

\section{Values of $\zeta$ at odd integers as multidimentional integrals on unit hypercubes} 
\label{formula}

Recall that the polylogarithm function of order $s\geq 1$ is defined for $z\in \{ z\in\mathbb{C}, |z|<1\}$ 
by ${\rm  Li}_s(z):=\sum_{k=1}^{+\infty}\frac{z^k}{k^s}$ (and is extented by analytic continuation to the whole complex plane). 

\bigskip

The aim of this section is to establish the following results (which we could not locate in the literature).

\begin{theorem}
\label{zetaodd}
Let $n$ be a positive integer. Then the value of the Riemann zeta function at odd integers 
is: 
\[ \zeta(2n+1)=(-1)^{n+1} \int\limits_{[0;1]^n} \left (\prod_{i=1}^{n}\frac{\log(x_i)}{x_i}\right ) \log\left (1-\prod_{i=1}^nx_i\right ) dx_1\cdots dx_n \]
\end{theorem}

\begin{proof}

Let $n$ be a positive integer, and for any integer $1\leq k\leq n$ define $D_{k,n}$ 
to be the set of all ordered $k$-tuples $j_1<\dots <j_k$ of distinct integers taken in $\{1,\dots ,n\}$. 
So $\# D_{k,n}=$ $n \choose k$. 
\bigskip

Define for any $(x_1,\dots ,x_n)\in ]0,1[^n$ (the open unit hypercube of dimension $n$) the function $M_n$ as 
\[ 
%\begin{split}
M_n(x_1,\dots ,x_n):= (-1)^{n+1} {\rm  Li}_{2n+1}(\prod_{u=1}^n x_u) + 
\sum_{i=1}^n (-1)^i \left ( \sum_{J\in D_{n-i,n}} \prod_{j\in J}\log(x_j)\right ) {\rm  Li}_{n+i}(\prod_{u=1}^n x_u) 
%\end{split}
\]

For any positive integer $n$, the function $M_n$ is at least $\mathcal{C}^{n}$ when $x_i\neq 0$ for all $i=1,\dots, n$. We are going to show that: 
\[
\frac{\partial^n}{\partial x_1\partial x_2 \cdots \partial x_n} M_n(x_1,\dots ,x_n)
=\left ( \prod_{i=1}^n\frac{\log(x_i)}{x_i}\right ) \log\left (1-\prod_{i=1}^nx_i\right )
\]

\bigskip

Differentiating $M_n$ with respect to $x_1$ one finds :

\[
\frac{\partial M_n}{\partial x_1}(x_1,\dots ,x_n)= (-1)^{n+1} \frac{ {\rm  Li}_{2n}(\prod_{k=1}^n x_k) }{x_1} +
\]
\[ 
\sum_{i=1}^n (-1)^i 
\Biggl( 
\biggl( \sum_{J\in D^{*(1)}_{n-i,n}} \frac{1}{x_1}\prod_{j\in J}\log(x_j)\biggr) {\rm  Li}_{n+i}(\prod_{k=1}^n x_k) 
+ \biggl( \sum_{J\in D_{n-i,n}} \prod_{j\in J}\log(x_j)\biggr) \frac{{\rm  Li}_{n+i-1}(\prod_{k=1}^n x_k)}{x_1} 
\Biggr) 
\]
where $D^{*(1)}_{n-i,n}$ denotes elements of the set $D_{n-i,n}$ where $1$ is not in the $(n-i)$-uplet, 
so this is also exactly the set of ordered $(n-i-1)$-uplets of distinct elements taken in the set $\{2,\dots ,n\}$, 
and so we have the inclusion $D^{*(1)}_{n-i,n} \subset D_{n-i-1,n}$.  
Hence the telescopic cancellations witnessed in the case $n=4$ occur between terms of two consecutive values of $i$. 
The remaining differentiations with respect to the other variables ultimately lead to the desired expression. 
Thus we have an explicit antiderivative of $\left (\prod_{i=1}^{n}\frac{\log(x_i)}{x_i}\right ) \log\left (1-\prod_{i=1}^nx_i\right )$ 
and it is clear from its expression that the generalized integral exists. Evaluating $M_n$ at $(x_1,\dots ,x_n)=(1,\dots ,1)$ finishes the proof.
\end{proof}
\bigskip

\begin{remark}
As remarked by the referee of a first version of this work:
\bigskip

(a) the formula in fact readily follows by 
using the entire series expansion $-\log\left (1-\prod_{i=1}^nx_i\right ) = \sum_{k=1}^{+\infty}\frac{1}{k}\left ( \prod_{i=1}^nx_i \right )^k$ (valid here 
as $\prod_{i=1}^nx_i\in (0,1)$) and 
a repeated use of Fubini's theorem with the integral of monomials $\int_0^1 x_i^{k-1}\log(x_i)dx_i=\frac{-1}{k^2}$ (though this would not provide the closed-form expression of the antiderivative) ;
\bigskip

(b) using the case $n=1$ and the change of variable $t=1-x$ one obtains 
\[
\zeta(3)=\int_0^1 \frac{\log(x)\log(1-x)}{x}dx=\int_0^1 \frac{\log(x)\log(1-x)}{1-x}dx
\]
and the formula of Janous stated in the introduction follows by adding those two integrals.
\end{remark}
\bigskip

\begin{corollary}
\label{zetaodd2}
Let $n$ and $r$ be positive integers. Then we have: 
\[ 
\frac{\zeta(2n+1)}{r^{2n}}=(-1)^{n+1} \int\limits_{[0;1]^n} \left (\prod_{i=1}^{n}\frac{\log(x_i)}{x_i}\right ) \log\left (1-\left(\prod_{i=1}^nx_i\right)^r \right ) dx_1\cdots dx_n \]
\end{corollary}

\begin{proof}
This follows from a change of variable $x\rightarrow x^r$ in the previous theorem
\footnote{This too was observed by the referee of the previous version of this work, the author had tediously proposed a proof along the lines of the theorem.}. 
\end{proof}
\bigskip

\begin{remark} By summing over $r\in\mathbb{N}^*$ this last formula we obtain:

\[
\zeta(2n)\zeta(2n+1)= \int\limits_{[0;1]^n} \left (\prod_{i=1}^{n}\frac{\log(x_i)}{x_i}\right ) \log\left (\phi \left(\prod_{i=1}^nx_i\right) \right ) dx_1\cdots dx_n 
\]

where $\phi$ is Euler's function defined for $q\in [0,1]$ by $\phi(q):=\prod_{n=1}^{+\infty}(1-q^n)$. 
\end{remark}
\bigskip

\begin{theorem}
\label{fractionzeta}
Let $n\geq 2$ and $k\geq 1$ be integers. Then we have the following closed-form expressions:

$(i)$ 
\[
\int\limits_{[0;1]^2} \frac{\log(x)}{x}\frac{\log(y)}{y}\log(1-(xy)^1) \log(x)\log(y) (xy)^{2k+1} dxdy = 
\]
\[
\frac{4\zeta(6)}{(2k+1)} + \frac{4\zeta(5)}{(2k+1)^2} + \frac{4\zeta(4)}{(2k+1)^3} + \frac{4\zeta(3)}{(2k+1)^4} + \frac{4\zeta(2)}{(2k+1)^5} 
-\frac{4}{(2k+1)^6}\sum_{n=1}^{2k+1}\frac{1}{n}-4\sum_{j=1}^{6}\left (\frac{1}{(2k+1)^j} \sum_{i=1}^{2k+1}\frac{1}{i^{7-j}}\right )
\]

$(ii)$  
\[
\int\limits_{[0;1]^2} \frac{\log(x)}{x}\frac{\log(y)}{y}\log(1-(xy)^2) \log(x)\log(y) (xy)^{2k+1} dxdy = 
\]
\[
\frac{63\zeta(6)}{8(2k+1)} + \frac{31\zeta(5)}{4(2k+1)^2} + \frac{15\zeta(4)}{2(2k+1)^3} + \frac{7\zeta(3)}{(2k+1)^4} + \frac{6\zeta(2)}{(2k+1)^5}
\]
\[
+\frac{8\log(2)}{(2k+1)^6}-\frac{4}{(2k+1)^6}\sum_{i=1}^{k}\frac{1}{2i+1}-8\sum_{j=1}^{6}\left (\frac{1}{(2k+1)^j}\sum_{i=0}^{k}\frac{1}{(2i+1)^{7-j}}\right )
\]

\end{theorem}
\begin{proof}
$(i)$ using the entire series expansion of $\log (1-xy)$ and Fubini's theorem we have, twice integrating by parts: 
\[
\int\limits_{[0;1]^2} \frac{\log(x)}{x}\frac{\log(y)}{y}\log(1-(xy)^1) \log(x)\log(y) (xy)^{2k+1} dxdy = 
\sum_{u=1}^{+\infty} \frac{-1}{u} \left ( \int_0^1 (\log(x))^2 x^{2k+1+u} dx\right ) \left ( \int_0^1 (\log(y))^2 y^{2k+1+u} dy\right )
\]
\[
= -\sum_{u=1}^{+\infty}\frac{4}{u} \frac{1}{(2k+1+u)^6} = -\frac{4}{(2k+1)^6}\sum_{u=1}^{+\infty}\left(\frac{1}{u}-\frac{1}{2k+1+u}\right ) 
+\frac{4}{(2k+1)^5}\sum_{u=1}^{+\infty}\frac{1}{(2k+1+u)^2} 
\]
\[+\frac{4}{(2k+1)^4}\sum_{u=1}^{+\infty}\frac{1}{(2k+1+u)^3} 
+ \frac{4}{(2k+1)^3}\sum_{u=1}^{+\infty}\frac{1}{(2k+1+u)^4}+\frac{4}{(2k+1)^2}\sum_{u=1}^{+\infty}\frac{1}{(2k+1+u)^5}
+\frac{4}{(2k+1)}\sum_{u=1}^{+\infty}\frac{1}{(2k+1+u)^6}
\]

and the result follows by adding and substracting the first terms in each sum to make the zeta values appear.
\bigskip

$(ii)$ the begining proceeds in a similar fashion to give:
\[
\int\limits_{[0;1]^2} \frac{\log(x)}{x}\frac{\log(y)}{y}\log(1-(xy)^2) \log(x)\log(y) (xy)^{2k+1} dxdy = 
\sum_{u=1}^{+\infty} \frac{-1}{u} \left ( \int_0^1 (\log(x))^2 x^{2k+1+2u} dx\right ) \left ( \int_0^1 (\log(y))^2 y^{2k+1+2u} dy\right )
\]
\[
= -\sum_{u=1}^{+\infty}\frac{4}{u} \frac{1}{(2k+1+2u)^6} = -\frac{4}{(2k+1)^6}\sum_{u=1}^{+\infty}\left(\frac{1}{2u}-\frac{2}{2k+1+2u}\right ) 
+\frac{4}{(2k+1)^5}\sum_{u=1}^{+\infty}\frac{2}{(2k+1+2u)^2} 
\]
\[+\frac{4}{(2k+1)^4}\sum_{u=1}^{+\infty}\frac{2}{(2k+1+2u)^3} 
+ \frac{4}{(2k+1)^3}\sum_{u=1}^{+\infty}\frac{2}{(2k+1+2u)^4}+\frac{4}{(2k+1)^2}\sum_{u=1}^{+\infty}\frac{2}{(2k+1+2u)^5}
+\frac{4}{(2k+1)}\sum_{u=1}^{+\infty}\frac{2}{(2k+1+u)^6}
\]

To conclude this time, we need a few more steps. First we use the well-known relation between sums on even and odd indices in zeta values, i.e. 
for any positive integer $m$ one has $\zeta(m)=\sum_{u=1}^{+\infty}\frac{1}{(2u)^m}+\sum_{u=0}^{+\infty}\frac{1}{(2u+1)^m}$ so that 
$\sum_{u=1}^{+\infty}\frac{1}{(2u+1)^m}=\left( 1-\frac{1}{2^m}\right ) \zeta(m) -1$.

Second, the sum multiplied by $\frac{-4}{(2k+1)^6}$ is no longer finite, and we use the opposite of the alternating harmonic series: 
$-\sum_{u=1}^{+\infty}\frac{(-1)^{u+1}}{u}=-\log(2)$.

\end{proof}
\bigskip

\begin{remark} 
In the previous version of this work, this had been stated as a conjecture (in a much less legible manner since we had not yet recognized 
$\zeta(2)$, $\zeta(4)$ and $\zeta(6)$ in it, and had added a third statement which we realized later is in fact completely erroneous). We are very grateful to the 
referee of a previous version of this work for suggesting that these expressions might be established using the same strategy she/he had mentioned 
about her/his alternative proof of theorem \ref{zetaodd} (entire series expansion and Fubini's theorem).
\end{remark}

\section{Proof of the irrationality of all $\zeta(2n+1)$ for $n\in\mathbb{N}^*$}
\label{irrational}

First let us recall the standard irrationality criteria of Dirichlet : 

\begin{lemma}(Dirichlet, 1848) 
$\alpha\in\mathbb{R}\backslash \mathbb{Q} \Leftrightarrow \forall \epsilon>0\quad \exists p\in\mathbb{N}\ \exists q\in\mathbb{Q} \text{ such that } |p\alpha-q| < \epsilon$.
\end{lemma}
\bigskip

In 1979, shortly after Ap\'ery presented his proof of the irrationality of $\zeta(2)$ and $\zeta(3)$, Beukers \cite{B} found another proof 
using Dirichlet's criteria applied to some particular integral representations of those two numbers. We quickly summarize the strategy 
as follows (the author also benefited from the extremely clear slides of Brown \cite{Br}).

\begin{itemize}
\item step 1: we have $\int_0^1\int_0^1 \frac{-\log(xy)}{1-xy} dxdy =2\zeta(3)$ and for any integer $r\geq 1$ 
we have $\int_0^1\int_0^1 \frac{-\log(xy)}{1-xy}(xy)^r dxdy =2\left (\zeta(3)-\frac{1}{1^3}-\cdots -\frac{1}{r^3}\right ) \leq 2\zeta(3)$ \\
\item step 2: by denoting a Legendre-type polynomial $P_k(x):=\frac{1}{k!}\left \{ \frac{d}{dx}\right \}^k x^k(1-x^k) \in\mathbb{Z}[X]$ and using the previous step 
we have 
$I_k:=\int_0^1\int_0^1 \frac{-\log(xy)}{1-xy}P_k(x)P_k(y) dxdy =\frac{A_k +B_k\zeta(3)}{d_n^3}$ with $A_k,B_k\in\mathbb{Z}$ and $d_k:=\text{lcm}(1,\dots ,k)$\\
\item step 3: by the Prime Number Theorem we have for any integer $k\geq 1$ that $d_k<3^k$ \\
\item step 4: we have $\int_0^1 \frac{1}{1-(1-xy)z}dz = -\frac{\log(xy)}{1-xy}$ \\
\item step 5: by integration by parts one finds also that 
$I_k= \int_0^1\int_0^1\int_0^1 \left ( \frac{x(1-x)y(1-y)z(1-z)}{1-(1-xy)z} \right )^k \frac{ dx dy dz}{1-(1-xy)z}$ \\
\item step 6: we can bound uniformly for $0\leq x,y,z\leq 1$ one part of the integrand $\frac{x(1-x)y(1-y)z(1-z)}{1-(1-xy)z}\leq (\sqrt{2}-1 )^4 < \frac{1}{2}$ 
(this is the reason for introducing $P_k$ rather than working with the integrand of step 1 where $(xy)^r$ can only be bounded by 1)\\
\item step 7: by using most of the previous steps we find $0<\left |   \frac{A_k +B_k\zeta(3)}{d_k^3} \right | \leq 2\zeta(3)(\sqrt{2}-1 )^{4k}$ \\
\item step 8: using now the information on the growth of $d_k$, so of $d_k^3$ too, we get $0<\left |   A_k +B_k\zeta(3)\right | \leq \left ( \frac{4}{5} \right )^k $, 
which concludes the proof by Dirichlet's criteria.
\end{itemize}

\bigskip

Unfortunately in the ensuing years and decades no tweak to that strategy could be made to work for values of $\zeta$ at other odd integers. 
Vasilyev \cite{V} could show that a direct generalization of Beukers's integral for $\zeta(5)$ can only show that one of $\zeta(3)$ and $\zeta(5)$ is irrational. 
In what follows we shall show that the expressions from the previous section do not seem to allow any progress on these matters. 
\bigskip

Define for positive integers $k$ the two following sequences of numbers: 
\[
I_k:=
\int\limits_{[0;1]^2} \frac{\log(x)}{x}\frac{\log(y)}{y}\log(1-(xy)) \log(x)\log(y) (xy)^{2k+1} dxdy 
\]
\[ 
 -\left (
\frac{4\zeta(6)}{(2k+1)}  + \frac{4\zeta(4)}{(2k+1)^3} + \frac{4\zeta(3)}{(2k+1)^4} + \frac{4\zeta(2)}{(2k+1)^5} 
\right )
\]
\[
=\frac{4\zeta(5)}{(2k+1)^2}-\frac{4}{(2k+1)^6}\sum_{n=1}^{2k+1}\frac{1}{n}-4\sum_{j=1}^{6}\left (\frac{1}{(2k+1)^j}\sum_{i=1}^{2k+1}\frac{1}{i^{7-j}}\right )
\]
\bigskip

and

\[
J_k:=
\int\limits_{[0;1]^2} \frac{\log(x)}{x}\frac{\log(y)}{y}\log(1-(xy)^2) \log(x)\log(y) (xy)^{2k+1} dxdy 
\]
\[ 
 -\left (
\frac{63\zeta(6)}{8(2k+1)}  + \frac{15\zeta(4)}{2(2k+1)^3} + \frac{7\zeta(3)}{(2k+1)^4} + \frac{6\zeta(2)}{(2k+1)^5}
+\frac{8\log(2)}{(2k+1)^6}
\right )
\]
\[
=\frac{31\zeta(5)}{4(2k+1)^2}-\frac{4}{(2k+1)^6}\sum_{i=1}^{k}\frac{1}{2i+1}-8\sum_{j=1}^{6}\left (\frac{1}{(2k+1)^j}\sum_{i=1}^{k}\frac{1}{(2i+1)^{7-j}}\right )
\]
\bigskip
 
Using the triangle inequality on the definition of $I_k$, and determining by a routine study of critical points that the supremum of the 
absolute value $| \log(x)\log(y) (xy)^{2k+1} |$ occurs at $x^*=y^*=e^{-1/(2k+1)}$,
we get the inequality:

\[
|I_k|\leq \underbrace{\text{Sup}\{ | \log(x)\log(y) (xy)^{2k+1} | \text{ where } 0\leq x,y\leq 1\}}_{ =\frac{1}{e^2(2k+1)^2} < \frac{1}{(2k+1)^2} } 
\times \underbrace{\left| \int\limits_{[0;1]^2} \frac{\log(x)}{x}\frac{\log(y)}{y}\log(1-(xy)) dxdy \right|}_{ = \zeta(5) }
\]
\[
+\frac{4\zeta(6)}{(2k+1)}  + \frac{4\zeta(4)}{(2k+1)^3} + \frac{4\zeta(3)}{(2k+1)^4} + \frac{4\zeta(2)}{(2k+1)^5} 
\]
\bigskip

which, multiplying both sides by $(2k+1)^2$, can be rewritten as:

\[
0 < \left| 4\zeta(5)- \frac{a_k}{b_k} \right | 
< \zeta(5) + 4(2k+1)\zeta(6)  + \frac{4\zeta(4)}{2k+1} + \frac{4\zeta(3)}{(2k+1)^2} + \frac{4\zeta(2)}{(2k+1)^3} 
\]
\bigskip

where $\frac{a_k}{b_k}:=(2k+1)^2\left (
\frac{4}{(2k+1)^6}\sum_{n=1}^{2k+1}\frac{1}{n}+4\sum_{j=1}^{6}\left (\frac{1}{(2k+1)^j}\sum_{i=1}^{2k+1}\frac{1}{i^{7-j}}\right ) \right) >0 $ is a certain irreducible rational number. 

\bigskip

Similarly with $J_k$ we ultimately have:

\[
0 < \left| 31\zeta(5)- \frac{c_k}{d_k} \right | 
< 4\zeta(5) + \frac{63(2k+1)\zeta(6)}{2}  + \frac{30\zeta(4)}{2k+1} + \frac{28\zeta(3)}{(2k+1)^2} + \frac{24\zeta(2)}{(2k+1)^3} + 
\frac{32\log(2)}{(2k+1)^4}
\]
\bigskip

where $\frac{c_k}{d_k}:=(4(2k+1)^2)\left (
\frac{4}{(2k+1)^6}\sum_{i=1}^{k}\frac{1}{2i+1}+8\sum_{j=1}^{6}\left (\frac{1}{(2k+1)^j}\sum_{i=0}^{k}\frac{1}{(2i+1)^{7-j}}\right )\right) >0 $ is a certain irreducible rational number.
\bigskip

In those two inequalities, to obtain irrationality of $\zeta(5)$ one would need that after multiplying all sides by $b_k$ (respectively $d_k$) 
the resulting right-hand side expression be a function that goes to $0$ as $k$ goes to infinity, which is not the case. 
\bigskip

This led the author to think about finding a function whose supremum on $[0,1]$ has a faster decay than $\mathcal{O}\left(\frac{1}{(2k+1)^2}\right)$. 
For any positive integer $m$, we can show in a similar fashion as before the closed-form expression of the following one-dimensional integral (where $s_{k,m}$ 
and $t_{k,m}$ are positive integers resulting from the combined completions of the sums):

\[
\int\limits_{[0;1]} \frac{\log(x)}{x}\log(1-x) (\log(x))^m x^{2k+1} dx = \sum_{u=1}^{+\infty}\frac{-(m+1)!}{u(u+2k+1)^{m+2}} 
= \sum_{j=0}^{m}\frac{(m+1)!\zeta(m+2-j)}{(2k+1)^{j+1}} -\frac{s_{k,m}}{t_{k,m}}
\]
\bigskip

Then defining the numbers: 

\[
Z_{k,m}:= \int\limits_{[0;1]} \frac{\log(x)}{x}\log(1-x) (\log(x))^m x^{2k+1} dx - \sum_{j=1}^{m}\frac{(m+1)!\zeta(m+2-j)}{(2k+1)^{j+1}}
=\frac{(m+1)!\zeta(m+2)}{2k+1} - \frac{s_{k,m}}{t_{k,m}}
\]
\bigskip

we obtain, after using the fact that the supremum of $|(\log(x))^mx^{2k+1}|$ occurs at $x^*=e^{-m/(2k+1)}$, and multiplying both sides by $2k+1$, that:

\[
0<\left|(m+1)!\zeta(m+2)-\frac{(2k+1)s_{k,m}}{t_{k,m}} \right | \leq \frac{m^m\zeta(3)}{(2k+1)^{m-1}e^m}+ \sum_{j=1}^{m}\frac{(m+1)!\zeta(m+2-j)}{(2k+1)^{j}}
\]
\bigskip

While the right hand side does tend to $0$ as $k$ goes to infinity, contrarily to what was falsely claimed in a previous version of this work, 
this does not prove the irrationality of $\zeta(m+2)$. For that, 
one would first need to multiply all sides by $t_{k,m}$ and that the resulting right-hand side be a function that goes to $0$ as $k$ goes to infinity. 
We have not computed $t_{k,m}$ explicitely, but it appears not to be the case.

\end{document}